\documentclass[10pt]{amsart}

\usepackage[T1]{fontenc}
\usepackage[utf8]{inputenc}
\usepackage{lmodern}
\usepackage{microtype}
\usepackage{mathtools}
\usepackage{amscd}
\usepackage{hyperref}

\hypersetup{hidelinks}

\newtheorem{theorem}{Theorem}[section]
\newtheorem{proposition}[theorem]{Proposition}
\newtheorem{lemma}[theorem]{Lemma}
\newtheorem{corollary}[theorem]{Corollary}
\theoremstyle{definition}

\theoremstyle{remark}
\newtheorem{remark}[theorem]{Remark}

\DeclareMathOperator{\Hom}{Hom}
\DeclareMathOperator{\End}{End}
\newcommand{\Sm}{\mathrm{Sm}}
\newcommand{\Spec}{\operatorname{Spec}}
\newcommand{\id}{\operatorname{id}}

\newcommand{\Ab}{\mathbf{Ab}}
\newcommand{\CH}{\operatorname{CH}}

\title[Dimension filtrations in birational localisation]
  {Dimension filtrations in birational localisation}

\author{David Kumallagov}
\address{}
\email{}


\begin{document}

\begin{abstract}
Let \(S_b\) be the class of birational morphisms between smooth
varieties over a field \(F\), and let \(L_n=S_b^{-1}d_{\leq n}\Sm(F)\).
Kahn and Sujatha asked whether the natural functor
\(L_n\to S_b^{-1}\Sm(F)\) is fully faithful. We prove that it is fully
faithful exactly for \(n=0\). More strongly, for every \(n\geq1\) and every
\(N\geq n+1\), the transition functor \(L_n\to L_N\) has an infinite fibre
on an endomorphism set. The proof identifies a sharp dimension threshold:
if \(\dim X+r\leq n\), then \(X\times\mathbb A^r\to X\) is invertible in
\(L_n\) precisely when \(\dim X+r\leq n-1\). We also give proper and
projective analogues.
\end{abstract}

\maketitle

\section{Introduction}

Let \(F\) be a field, and let \(\Sm(F)\) denote the category of smooth,
integral, separated schemes of finite type over \(F\). We silently replace
each essentially small category by a fixed small skeleton before localising;
this does not change any of the Hom-sets considered below. We use
Gabriel--Zisman localisations of categories \cite{GabrielZisman}. We write
\(S_b\) for the class of birational morphisms in \(\Sm(F)\). For
\(n\geq0\), let
\(d_{\leq n}\Sm(F)\) be the full subcategory formed by varieties of
dimension at most \(n\), and put
\[
  L_n=S_b^{-1}d_{\leq n}\Sm(F),
  \qquad
  L_\infty=S_b^{-1}\Sm(F).
\]
For \(n\leq N\leq\infty\), inclusion induces a transition functor
\[
  \rho_{n,N}\colon L_n\longrightarrow L_N.
\]

Kahn and Sujatha asked whether \(\rho_{n,\infty}\) is fully faithful
\cite[Question 8.8(3)]{KahnSujatha}. Also in
\cite[Remark 8.12]{KahnSujathaLocalisation}, the same authors stated that the transition between the
corresponding smooth projective categories is fully faithful in
characteristic zero. Our results show that the
projective assertion announced there cannot hold as stated, already for
\(n=1\); see Remark~\ref{rem:earlier-announcement}.
We answer the later question and obtain a stronger statement.

\begin{theorem}\label{thm:main}
Let \(F\) be any field.
\begin{enumerate}
\item The functor \(\rho_{0,\infty}\) is fully faithful.
\item For every \(n\geq1\) and every \(N\) with \(n+1\leq N\leq\infty\),
      the functor \(\rho_{n,N}\) is not faithful. In fact, the map
      \[
        \End_{L_n}(\mathbb A^n_F)\longrightarrow
        \End_{L_N}(\mathbb A^n_F)
      \]
      has an infinite fibre.
\end{enumerate}
Consequently, \(L_n\to L_\infty\) is fully faithful if and only if
\(n=0\).
\end{theorem}

\begin{remark}\label{rem:earlier-announcement}
This also shows that the assertion announced without proof in
\cite[Remark 8.12]{KahnSujathaLocalisation} cannot hold as stated. Indeed, in
the birational localisation of
smooth projective varieties of dimension at most one, birational morphisms of
smooth projective curves are already isomorphisms, so the self-maps
\([u:v]\mapsto [u^m:v^m]\) of \(\mathbb P^1\) remain pairwise distinct. After
allowing smooth projective surfaces, the Colliot--Th\'el\`ene surface argument
makes \(\mathbb P^1\to\operatorname{Spec}F\) invertible, and all these
self-maps become the identity. Thus the dimension-one to dimension-two
transition is not faithful over any field.
\end{remark}

The mechanism is elementary. The category \(L_n\) admits a
``top-dimensional degree'' functor which vanishes below dimension \(n\)
and sends a dominant endomorphism of an \(n\)-fold to its total generic
degree. This detects morphisms which cannot yet be identified inside
\(L_n\). The upper bound in our dimension threshold is obtained by keeping
track of dimensions in the surface argument of Colliot-Th\'el\`ene
\cite[Appendix A]{KahnSujatha}; the degree functor proves that this bound is
optimal.

More precisely, we prove the following sharp form of homotopy invariance.

\begin{theorem}\label{thm:threshold-intro}
Let \(X\in\Sm(F)\) have dimension \(d\), and let \(r\geq1\) and
\(d+r\leq n\). The projection
\[
  X\times\mathbb A^r\longrightarrow X
\]
is an isomorphism in \(L_n\) if and only if \(d+r\leq n-1\).
\end{theorem}

Thus the projection first becomes invertible one step after the dimension
of its source. This exact threshold explains why every consecutive map
\(L_n\to L_{n+1}\) already fails to be faithful. We also prove the same
phenomenon in the smooth proper and smooth projective subcategories; there
the degree detector is simply the top-dimensional Chow group. The full
birational localisation and its relation to \(R\)-equivalence and
\(\mathbb A^1\)-homotopy have recently been revisited in
\cite{CisinskiKahn,Koizumi}; the point here is the unstable behaviour caused
by imposing a dimension bound.

\section{The top-dimensional degree}

We first construct the obstruction used throughout the paper.

\begin{lemma}\label{lem:degree}
For every \(n\geq1\), there is a functor
\[
  \Delta_n\colon d_{\leq n}\Sm(F)\longrightarrow\Ab
\]
defined on objects by
\[
  \Delta_n(X)=
  \begin{cases}
    \mathbb Z,&\dim X=n,\\
    0,&\dim X<n,
  \end{cases}
\]
and on a morphism \(f\colon X\to Y\) as follows. If \(X\) and \(Y\) both
have dimension \(n\) and \(f\) is dominant, then \(f\) is generically
finite and \(\Delta_n(f)\) is multiplication by the total field degree
\([F(X):F(Y)]\). In all other cases \(\Delta_n(f)=0\).

The functor \(\Delta_n\) sends every birational morphism to an
isomorphism. Hence it factors uniquely through a functor
\[
  \overline\Delta_n\colon L_n\longrightarrow\Ab.
\]
\end{lemma}

\begin{proof}
The first part is obvious. The rest follows from the universal property of localization.
\end{proof}

\begin{corollary}\label{cor:degree-separates}
Let \(f,g\colon X\to Y\) be morphisms between smooth \(n\)-folds. If their
generic degrees are different (where a non-dominant morphism is assigned
degree zero), then \(f\) and \(g\) define different morphisms in \(L_n\).
\end{corollary}

The next lemma keeps track of the dimension used in the classical argument
which makes \(\mathbb P^1\) birationally contractible.

\begin{lemma}\label{lem:P1}
Let \(X\in\Sm(F)\) have dimension \(d\). The projection
\[
  q_X\colon X\times\mathbb P^1\longrightarrow X
\]
is an isomorphism in \(L_{d+2}\).

If \(X\) is proper (respectively projective), the same assertion holds in
the localisation of smooth proper (respectively smooth projective),
varieties of dimension at most \(d+2\).
\end{lemma}

\begin{proof}
We repeat the arguments of Colliot-Th\'el\`ene
\cite[Appendix A]{KahnSujatha}. Let \(W\) be the smooth projective surface
which is simultaneously the blow-up of \(\mathbb P^1\times\mathbb P^1\)
at an \(F\)-rational point and the blow-up of \(\mathbb P^2\) at two
\(F\)-rational points. Thus there are birational morphisms
\[
  W\longrightarrow\mathbb P^1\times\mathbb P^1,
  \qquad
  W\longrightarrow\mathbb P^2.
\]
One of the curves contracted by \(W\to\mathbb P^2\), say
\(E\simeq\mathbb P^1\), is mapped isomorphically by
\(W\to\mathbb P^1\times\mathbb P^1\) onto a ruling
\(L\simeq\mathbb P^1\), and is contracted to an \(F\)-rational point
of \(\mathbb P^2\).

Multiply this configuration by \(X\). All objects obtained are smooth and
have dimension at most \(d+2\), while the two displayed morphisms remain
birational. Let \(T\) be any object of \(L_{d+2}\), and put
\[
  \Phi_T(-)=\Hom_{L_{d+2}}(-,T).
\]
The inclusion of the ruling \(X\times L\) into
\(X\times\mathbb P^1\times\mathbb P^1\) induces a surjection on \(\Phi_T\):
one of the two projections
\[
  X\times\mathbb P^1\times\mathbb P^1\longrightarrow
  X\times\mathbb P^1
\]
restricts to an isomorphism on \(X\times L\), and therefore supplies a
right inverse after applying the contravariant functor \(\Phi_T\). Since
\(X\times W\to X\times\mathbb P^1\times\mathbb P^1\) is birational and
\(X\times E\to X\times L\) is an isomorphism, the resulting commutative
square shows that
\[
  \Phi_T(X\times W)\longrightarrow\Phi_T(X\times E)
\]
is surjective.

The birational morphism \(X\times W\to X\times\mathbb P^2\) makes
\(\Phi_T(X\times\mathbb P^2)\to\Phi_T(X\times W)\) a bijection. Hence the
composite
\[
  \Phi_T(X\times\mathbb P^2)\longrightarrow
  \Phi_T(X\times W)\longrightarrow
  \Phi_T(X\times E)
\]
is surjective. On the other hand, the commutative diagram
\[
\begin{CD}
  X\times E @>>> X\times W\\
  @VVV              @VVV\\
  X\times\{M_1\} @>>> X\times\mathbb P^2
\end{CD}
\]
shows that this composite factors as
\[
  \Phi_T(X\times\mathbb P^2)\longrightarrow
  \Phi_T(X\times\{M_1\})\longrightarrow
  \Phi_T(X\times E).
\]
It follows that
\[
  \Phi_T(X)\longrightarrow\Phi_T(X\times E)
\]
is surjective, because \(X\times\{M_1\}\simeq X\). It is also injective:
the projection \(X\times E\to X\) has a section, so the induced map on
\(\Phi_T\) has a left inverse. Thus it is a bijection for every \(T\).
Since \(E\simeq\mathbb P^1\), Yoneda's lemma shows that \(q_X\) is an
isomorphism.

The same proof works in
the proper and projective subcategories.
\end{proof}

\begin{theorem}\label{thm:threshold}
Let \(X\in\Sm(F)\) have dimension \(d\), and let \(r\geq1\) with
\(d+r\leq n\). Then the projection
\[
  p_{X,r}\colon X\times\mathbb A^r\longrightarrow X
\]
is an isomorphism in \(L_n\) if and only if \(d+r\leq n-1\).
\end{theorem}

\begin{proof}
Assume first that \(d+r\leq n-1\). For \(1\leq j\leq r\), put
\(X_j=X\times\mathbb A^{j-1}\). Then
\(\dim X_j=d+j-1\), and Lemma \ref{lem:P1} shows that
\[
  X_j\times\mathbb P^1\longrightarrow X_j
\]
is invertible in \(L_{d+j+1}\), hence in \(L_n\), because
\(d+j+1\leq d+r+1\leq n\). The open immersion
\(X_j\times\mathbb A^1\hookrightarrow X_j\times\mathbb P^1\) is
birational and therefore invertible in \(L_n\). It follows that
\(X_j\times\mathbb A^1\to X_j\) is invertible in \(L_n\). Iterating over
\(j\) proves that \(p_{X,r}\) is invertible.

If \(d+r=n\), then
\[
  \Delta_n(X\times\mathbb A^r)=\mathbb Z,
  \qquad
  \Delta_n(X)=0.
\]
Since \(\Delta_n\) factors through \(L_n\), the two objects cannot be
isomorphic there. Hence \(p_{X,r}\) is not invertible.
\end{proof}

\begin{remark}\label{rem:height}
Equivalently, among integers \(N\geq d+r\), the smallest \(N\) for which
\(X\times\mathbb A^r\to X\) becomes invertible in \(L_N\) is
\(N=d+r+1\). The same statement holds with \(\mathbb A^1\) replaced by
\(\mathbb P^1\) when \(r=1\).
\end{remark}

\section{Failure of faithfulness}

We now prove the negative part of Theorem \ref{thm:main} in a form which
exhibits infinitely many morphisms in one fibre.

\begin{theorem}\label{thm:infinite-fibre}
Let \(n\geq1\), and let \(N\) satisfy \(n+1\leq N\leq\infty\). For
\(m\geq1\), define
\[
  \mu_m\colon\mathbb A^n\longrightarrow\mathbb A^n,
  \qquad
  (x_1,\ldots,x_n)\longmapsto
  (x_1,\ldots,x_{n-1},x_n^m).
\]
The morphisms \(\mu_m\) are pairwise distinct in \(L_n\), but all have the
same image, namely the identity, in \(L_N\). Consequently
\(\rho_{n,N}\) is not faithful and has an infinite fibre on
\(\End(\mathbb A^n)\).
\end{theorem}

\begin{proof}
The generic degree of \(\mu_m\) is \(m\). Indeed, after adjoining the
first \(n-1\) coordinates, put \(K=F(x_1,\ldots,x_{n-1})\) and
\(u=x_n^m\). The polynomial \(T^m-u\in K[u][T]\) is Eisenstein at
\(u\), so \([K(x_n):K(u)]=m\). Thus the degree statement remains
valid in positive characteristic, including when the extension is
inseparable. Corollary
\ref{cor:degree-separates} therefore shows that the \(\mu_m\) are pairwise
distinct in \(L_n\).

Let \(p\colon\mathbb A^n\to\mathbb A^{n-1}\) be the projection onto the
first \(n-1\) coordinates, with \(\mathbb A^0=\Spec F\). Theorem
\ref{thm:threshold}, applied in \(L_{n+1}\), shows that \(p\) is an
isomorphism. Since \(p\mu_m=p\), we obtain \(\mu_m=\id_{\mathbb A^n}\) in \(L_{n+1}\).

The same equality holds after applying any transition functor
\(L_{n+1}\to L_N\).
\end{proof}

\begin{remark}\label{rem:general-affine}
The affine space \(\mathbb A^{n-1}\) is not essential here. If \(Z\) is any
smooth \((n-1)\)-fold, then the endomorphisms
\[
  \id_Z\times(t\longmapsto t^m)\colon
  Z\times\mathbb A^1\longrightarrow Z\times\mathbb A^1,
  \qquad m\geq1,
\]
are pairwise distinct in \(L_n\) and all become the identity in
\(L_{n+1}\). The proof is identical: \(\Delta_n\) records the generic
degree \(m\), while the projection \(Z\times\mathbb A^1\to Z\) is
invertible in \(L_{n+1}\).
\end{remark}

There is also an immediate additive strengthening. If \(R\) is a
commutative ring and \(\mathcal C\) is a category, write \(R[\mathcal C]\)
for the free \(R\)-linear category on \(\mathcal C\).

\begin{corollary}\label{cor:linear}
Let \(R\neq0\) be a commutative ring. For \(n\geq1\) and
\(N\geq n+1\), the kernel of
\[
  \End_{R[L_n]}(\mathbb A^n)\longrightarrow
  \End_{R[L_N]}(\mathbb A^n)
\]
contains a free \(R\)-submodule of countable rank, generated by
\([\mu_m]-[\mu_1]\) for \(m\geq2\).
\end{corollary}

\begin{proof}
The \([\mu_m]\) are distinct basis vectors in the free \(R\)-module
\(\End_{R[L_n]}(\mathbb A^n)\), while Theorem
\ref{thm:infinite-fibre} identifies all of them in \(R[L_N]\).
\end{proof}

It remains to treat dimension zero. 

\begin{proposition}\label{prop:zero}
The functor \(\rho_{0,\infty}\colon L_0\to L_\infty\) is fully faithful.
\end{proposition}

\begin{proof}
Every object of \(d_{\leq0}\Sm(F)\) is \(\Spec K\) for a finite separable
extension \(K/F\), and every birational morphism between such objects is
an isomorphism. Hence \(L_0=d_{\leq0}\Sm(F)\).

Let \(U=\Spec K\) and \(V=\Spec E\) be two such objects. The target \(V\)
is smooth and proper. The Hom formula of Kahn and Sujatha
\cite[Theorem 6.6.3]{KahnSujatha} gives
\[
  \Hom_{L_\infty}(U,V)\simeq V(K)/R, 
\]
where the quotient on the right is taken with respect to the classical
\(R\)-equivalence of Manin \cite[Ch.~II, \S14]{Manin}; see also \cite[Definition 6.6.1]{KahnSujatha}.
The relation \(R\) is equality on \(V(K)\). Indeed, a rational map
\(\mathbb P^1_K\dashrightarrow V_K\), defined at two prescribed points,
extends uniquely to a morphism because \(\mathbb P^1_K\) is a regular
curve and \(V_K\) is proper. More explicitly, writing \(V_K=\Spec A\),
such a morphism is equivalent to
a \(K\)-algebra homomorphism
\(A\to\Gamma(\mathbb P^1_K,\mathcal O)=K\), and hence factors through a
single \(K\)-point of \(V_K\). Thus every direct \(R\)-equivalence is
trivial, and therefore the equivalence relation generated by direct
\(R\)-equivalence is equality. Consequently,
\[
  V(K)/R=V(K)=\Hom_F(\Spec K,\Spec E).
\]
This identification is the map induced by \(L_0\to L_\infty\), proving
full faithfulness.
\end{proof}

Theorem \ref{thm:main} follows from Theorem \ref{thm:infinite-fibre} and
Proposition \ref{prop:zero}.

\section{Proper and projective variants}

Let \(L_n^{\mathrm{prop}}\), respectively \(L_n^{\mathrm{proj}}\), denote
the birational localisation of the full subcategory of smooth proper,
respectively smooth projective, varieties of dimension at most \(n\). We
also use the evident notation \(L_\infty^{\mathrm{prop}}\) and
\(L_\infty^{\mathrm{proj}}\).

\begin{theorem}\label{thm:proper}
Let \(n\geq1\) and \(n+1\leq N\leq\infty\). The transition functors
\[
  L_n^{\mathrm{prop}}\longrightarrow L_N^{\mathrm{prop}},
  \qquad
  L_n^{\mathrm{proj}}\longrightarrow L_N^{\mathrm{proj}}
\]
are not faithful and have infinite fibres on suitable endomorphism sets.
More precisely, on
\(Y=\mathbb P^{n-1}\times\mathbb P^1\), the endomorphisms
\[
  \nu_m=\id_{\mathbb P^{n-1}}\times
  \bigl([u:v]\longmapsto[u^m:v^m]\bigr),
  \qquad m\geq1,
\]
are pairwise distinct in \(L_n^{\mathrm{proj}}\) and in
\(L_n^{\mathrm{prop}}\), but become equal to \(\id_Y\) in dimension
\(n+1\), hence in dimension \(N\).
\end{theorem}

\begin{proof}
The morphism \(\nu_m\) has generic degree \(m\), so the restriction of
\(\Delta_n\) separates the \(\nu_m\) in both dimension-\(n\)
localisations. By the proper (resp. projective) part of Lemma \ref{lem:P1}, the
projection
\[
  Y=\mathbb P^{n-1}\times\mathbb P^1
  \longrightarrow\mathbb P^{n-1}
\]
is invertible in both dimension-\(n+1\) localisations. It equalises all
\(\nu_m\), and therefore each \(\nu_m\) becomes the identity, exactly as
in the proof of Theorem \ref{thm:infinite-fibre}.
\end{proof}

\begin{remark}\label{rem:general-proper}
The same argument works with \(\mathbb P^{n-1}\) replaced by any smooth
proper, respectively smooth projective, \((n-1)\)-fold \(Z\): the maps
\[
  \id_Z\times\bigl([u:v]\longmapsto[u^m:v^m]\bigr)
\]
on \(Z\times\mathbb P^1\) are pairwise distinct in
\(L_n^{\mathrm{prop}}\), respectively \(L_n^{\mathrm{proj}}\), and all
become the identity in the corresponding localisation at level \(n+1\).
\end{remark}

\begin{remark}\label{rem:chow}
On smooth proper varieties of dimension at most \(n\), the ''detector''
\(\Delta_n\) is nothing more than \(\CH_n(-)\). Indeed,
\(\CH_n(X)=\mathbb Z[X]\) for an integral proper \(n\)-fold and is zero
below dimension \(n\); proper push-forward multiplies the fundamental
class by the generic degree. Thus the proper obstruction is the most
basic possible Chow-theoretic one.
\end{remark}

\end{document}